\documentclass[reqno]{amsart}
\usepackage{latexsym,amssymb,color}
\usepackage{upref}

\theoremstyle{plain}
\newtheorem{theorem}{Theorem}
\newtheorem{lemma}[theorem]{Lemma}

\newtheorem{proposition}[theorem]{Proposition}
\newtheorem{example}[theorem]{Example}
\theoremstyle{definition}
\newtheorem{definition}[theorem]{Definition}
\theoremstyle{remark}
\newtheorem{remark}[theorem]{Remark}

\DeclareMathOperator{\Div }{div}

\def\ds{\displaystyle}
\def\sgn{{\rm sgn}}

\newcommand{\eps}{\varepsilon}

\def\R{I\!\!R}
\def\Z{\mathbb{Z}}
\def\N{I\!\!N}

\def\pa{\partial}

\def\cal{\mathcal}

\def\b{\infty}
\def\la{\lambda}

\def\epsm{\eps_m}

\def\be{\begin{equation}}
\def\ee{\end{equation}}

\begin{document}

\title[Strong traces for transport equations]{Strong traces for averaged solutions of heterogeneous ultra-parabolic transport equations}
\date{\today}
\author{Jelena Aleksi\'c}
          \address{Jelena Aleksi\'c, Department of Mathematics and Informatics,
              University of Novi Sad, Trg D. Obradovi\'ca 4, 21000 Novi Sad,
              Serbia. }
          \email{jelena.aleksic@dmi.uns.ac.rs}
\author{Darko Mitrovic}
           \address{Darko Mitrovic,
              Faculty of Mathematics and Natural Sciences, University of
              Montenegro, Cetinjski put bb, 81000 Podgorica, Montenegro, and}
              \address{
              Faculty of Mathematics, University of
              Bergen, Johannes Bruns gate 12, 5008 Bergen, Norway}
              \thanks{The corresponding author: Darko Mitrovic, matematika@t-com.me }
           \email{matematika@t-com.me }

\begin{abstract}
       We prove that if traceability conditions are fulfilled then a weak solution $h\in L^\infty(\R^+\times\R^d\times \R)$ to {the ultra-parabolic transport equation}
       \begin{equation*}
             \pa_t h + \Div_x \left(F(t,x,\lambda)h\right)=\sum\limits_{i,j=1}^k
             \pa^2_{x_i x_j}\left(b_{ij}(t,x,\lambda) h\right)+\pa_\lambda
             \gamma(t,x,\lambda),
       \end{equation*}
       is such that for every $\rho\in C^1_c(\R)$, the velocity averaged quantity $\int_{\R}h(t,x,\lambda)$ $\rho(\lambda)d\lambda$
       admits the strong $L^1_{\rm loc}(\R^d)$-limit as $t\to 0$,
       i.e. there exist $h_0(x,\lambda)\in L^1_{\rm loc}(\R^d\times \R)$ and set $E\subset\R^+$ of full measure such that for 
       every $\rho\in C^1_c(\R)$,
       $$
         L^1_{\rm loc}(\R^d)-\lim\limits_{t\to 0, \; t\in E}
         \int_{\R} h(t,x,\lambda)\rho(\lambda)d\lambda= \int_{\R} h_0(x,\lambda) \rho(\lambda)d\lambda.
       $$
       As a corollary, under the traceability conditions, we prove existence of strong traces for entropy solutions to ultraparabolic equations in heterogeneous media.
\end{abstract}


\maketitle

\section{Introduction}

The aim of this paper is to investigate behavior of the averaged
        quantity $\int_{\R} \rho(\lambda)$ $h(t,x,\lambda)d\lambda$, $\rho \in C^1_c(\R)$, as $t\to 0$, where $h\in
        L^\infty(\R^+\times \R^d \times \R)$ is a weak solution to the
        equation
\begin{equation}\label{opsta}
\pa_t h + \Div_x \left(F(t,x,\lambda)h\right) =\sum\limits_{i,j=1}^k
\pa^2_{x_i x_j}\left(b_{ij}(t,x,\lambda)  h\right)+\pa_\lambda
 \gamma(t,x,\lambda),
\end{equation} where $(t,x)\in \R^d_+:=\R^+\times\R^d \equiv (0,\b)\times\R^d$,
$k\leq d$, $k,d\in\N$, and $\lambda \in \R$.

         The coefficients of the equation satisfy the following assumptions:
        \begin{itemize}
           \item $F\in C^1(\R^d_+\times\R;\R^d)$ and $b_{ij}\in C^1(\R^d_+\times\R; \R)$, $i,j=1,...k$;
           \item The matrix $b(t,x,\lambda)=[b_{ij}(t,x,\lambda)]_{i,j=1,\dots,k}$
                 is nonnegative definite in the sense that
                 \begin{equation}\label{3}
                     \langle b(t,x,\lambda) \xi, \xi\rangle\geq c(\lambda) |\xi|^2,  \quad \xi\in \R^{k}, \; \la\in\R,
                 \end{equation}
                 where $\langle\cdot,\cdot\rangle$ denotes the scalar product in $\R^k$.
                 The nonnegative continuous function $c$ fulfills the following:
                 there exists increasing sequence of real numbers $\{\lambda_i \}_{i\in \Z}$, such that
                 \begin{equation}\label{4}
                        c(\lambda)>0, \ \ {\rm for} \ \ \lambda \in
                        \bigcup_{i=-\b}^{+\infty} (\lambda_{i},\lambda_{i+1});
                 \end{equation}
           Moreover, the elements of the matrix $b$ are of the following form,
                 \begin{equation}\label{2h}
                        b_{ij}(t,x,\lambda)=\sum_{l=1}^k \sigma^b_{il}(t,x,\lambda)\sigma^b_{lj}(t,x,\lambda), \ \ i,j=1,..,k,
                 \end{equation}
                 where $\sigma^b_{il}(t,x,\lambda)=\sigma^b_{li}(t,x,\lambda)$.

           \item $\gamma \in {\cal M}(\R^d_+\times\R) $ is a locally finite Borel measure.
\end{itemize}

We consider bounded solutions $h$ being equal to zero if
$\lambda\notin (-M, M)$ for some $M>0$, and such that $\pa_{x_i}
h\in L^2(\R_+^d; {\cal M}(\R))$, $i=1,\dots, k$, i.e. we assume:
           \begin{equation}\begin{split}
           \label{assum}
           h\in L^\b(\R^{d+1}_+), & \quad (\exists M>0)\; \lambda\notin (-M, M) \implies h(t,x,\la)=0 \mbox{ and} \\
          & (\forall  i=1,\dots,k)\; \|\pa_{x_i} h\|_{{\cal M}(\R_\lambda)}  \in
           L_{\rm loc}^2(\R^d_+).
          \end{split}\end{equation}
In Section 5 we will see that \eqref{assum} is fulfilled for
kinetic functions corresponding to entropy solutions to
ultra-parabolic equations.\\

We shall prove that, under traceability conditions (see Definition
\ref{d-assump} and Definition \ref{d11}) which essentially provide
us to somehow remove non-degenerate heterogeneity from the flux (see
\eqref{kingnl}),
$h$ has the following property:

\begin{definition}\label{avsubt}
We say that the function $h$ {admits an averaged trace} if there
exists a function $h_0\in L^\infty(\R^d\times \R)$ and a set
$E\subset \R^+$ of full measure such that
       \begin{equation}
       \label{res_short}
       \lim\limits_{\substack{t\to 0\\ t\in E}}\underset{K}{\;\;\int}  \Big|\int_{\R} \left(  h(t,x,\lambda)-h_0(x, \la )\right) \rho(\la) \, d\la \Big|\,  dx=0,
       \end{equation} for any function $\rho\in C^1_c(\R)$ and any relatively compact $K\subset\subset \R^d$.
\end{definition}


Equation \eqref{opsta} describes transport processes in
heterogeneous media in which the diffusion (represented by the
second order terms) can be neglected in certain directions, cf.
\cite{3}. It is linear, but since it has derivatives of a measure on
the right-hand side, it describes entropy solutions to
ultra-parabolic equations (this is a special case of \cite{PC}).
Such equations were firstly considered by Graetz \cite{4}, and
Nusselt \cite{5}, in the investigations concerning the heat
transfer. Moreover, equations of type \eqref{opsta} describe
processes in porous media (cf. \cite{saz-x}), such as oil extraction
or ${\rm CO}_2$ sequestration which typically occur in highly
heterogeneous surroundings. One can also find applications in
sedimentation processes, traffic flow, radar shape-from-shading
problems, blood
          flow, gas flow in a variable duct and so on.

The question of existence of traces was firstly raised in the
context of limit of {hyperbolic relaxation toward a scalar
conservation laws} (see e.g. \cite{Nat, Tsa}). After that, a few
interesting papers appeared from the viewpoint of obtained results
and from the viewpoint of developed techniques \cite{VK, pan_trace,
Pan1, vass}. Proofs were based on reducing scalar conservation laws
on a transport equation (kinetic formulation \cite{LPT}), and then
using velocity averaging results. Here, we shall start with a
transport equation and prove that such equations satisfy the trace
property in the sense of \eqref{res_short}. In the basis of our
procedure are the classical blow up techniques, \cite{8, pan_trace,
vass}, velocity averaging results \cite{LM}, and induction with
respect to the space dimension \cite{pan_trace}.

All previous results on traces were given for the scalar
conservation laws in homogeneous media (see e.g. \cite{VK, Pan1} and
references therein). However, heterogenous framework is much more
natural. For instance, if we have the Hele-Shaw cell and we fill it
with sand (homogeneous sand), no matter how careful we are, we
cannot get a homogeneous medium (permeability is changed up to a
factor 2 from point to point) \cite{JMN}. That means that fluxes in
equations describing phenomena in such media must depend on the
space variable.


The paper is organized as follows.

First, in Section 3,  we consider the case when the components
$F_{k+1},...,F_d$ of the flux function $F=(F_1,..., F_d)$ depend
only on the variable $\la$ (without explicit dependence on $t$ and
$x$), i.e. we consider the equation
           \begin{align}\label{homogena}
               \pa_t h+   &  \sum_{i=k+1}^d \pa_{x_i}(F_{i}(\lambda) h) - \sum_{i,j=1}^k  \pa^2_{x_i x_j}(b_{ij}(t,x,\lambda)  h) = \\
                  &=
                  \pa_\lambda \gamma(t,x,\lambda)- \sum_{i=1}^k \pa_{x_i}(F_{i}(t,x,\lambda)h),\nonumber
           \end{align}
 in $ {\cal D}'(\R^d_+\!\times\! \R)$. We prove that bounded weak solution $h$ of \eqref{homogena}  that satisfies \eqref{assum} admits an averaged trace in the sense of Definition \ref{avsubt}.


In Section 4, under additional assumptions on the flux which we
called traceability, we shall use appropriate change of variables to
reduce equation \eqref{opsta} on an equation of type
\eqref{homogena} in a neighborhood of every point where the
existence of traces could be lost.

In Section 5, under the traceability conditions, we shall prove
existence of traces for entropy solutions to ultra-parabolic
equations. In particular, this includes entropy solutions to some examples of
heterogeneous scalar conservation laws.\\

\section{Auxiliary results}


First, we prove existence of a weak trace.

      \begin{proposition}\label{wt}
              If $h\in L^\b(\R^{d}_+\times \R)$ is a distributional solution to \eqref{opsta} satisfying \eqref{assum},
              then there exists $h_0\in L^\b(\R^{d+1})$, such that
                                $$
                                  h(t,\cdot, \cdot) \rightharpoonup h_0,  \mbox{ \rm weakly-$\star$ in } L^\b(\R^{d+1}),
                                  \mbox{ \rm as } t\to 0,\; t\in E,
                                $$
              where $ E:= \{t>0\,|\,(t,x,\la) \mbox{ \rm is a Lebesgue point to }
              h(t,x,\lambda) \mbox{ \rm for a.e. } (x,\lambda)\in\R^{d+1} \}$.

Moreover, there exists a zero sequence $\eps_m$ such that for almost
every $y\in \R^d$, it holds $h_0(
\sqrt{\eps_m}\bar{x}+\eps_m\tilde{x}+y,\lambda)\to h_0(y,\lambda)$,
strongly in $L^1_{\rm
loc}(\R^d\times \R)$ as $m\to \infty$, where $\bar{x}=(x_1,...,x_k,0,...,0)$ and $\tilde{x}=(0,...,0,x_{k+1},...,x_d)$.
      \end{proposition}

      \begin{proof}
             Since $h\in L^\b(\R^d_+\times \R)$, the family $\{ f(t,\cdot,\cdot)\}_{t\in E}$ is bounded in $L^\b(\R^{d+1})$.
             Due to weak-$\star$ precompactness of $L^\b(\R^{d+1})$, there exists a sequence $\{t_m\}_{m\in \N}$, $t_m\to 0$, as $m\to \b$,
             and $h_0\in L^\b(\R^{d+1})$, such that
                         \begin{equation}\label{wc}
                                  h(t_m, \cdot,\cdot) \rightharpoonup h_0(\cdot,\cdot), \mbox{ weakly-$\star$ in $L^\b(\R^{d+1})$, as } m\to \b.
                         \end{equation}
             For $\phi\in C_c^\b(\R^d)$, $\rho\in C^1_c(\R)$, denote
             $$
               I(t):=\int_{\R^{d+1}}h(t,x,\lambda)\rho(\lambda)\phi(x)\,dxd\lambda, \quad t\in E.
             $$
             With this notation, \eqref{wc} means that
                         \begin{equation}\label{wc1}
                                \lim_{m\to \b} I(t_m)= \int_{\R^{d+1}} h_0(x,\lambda)\rho(\lambda)\phi(x)\, dx d\lambda = :I(0).
                         \end{equation}
             Now, fix $\tau\in E$ and $m_0\in \N$, such that $t_m\leq \tau$, for $m\geq m_0$, to obtain

             \begin{equation}\begin{split}\nonumber
                                      & I(\tau)- I(t_m) = \int_{t_m}^\tau I'(t)\,dt = \int_{t_m}^\tau \int_{\R^{d+1}} \pa_t h(t,x,\la) \rho(\la) \phi(x)\,dx d\la \,dt\\
                                      &= \sum_{i=1}^d \int_{(t_m,\tau] \times \R^{d+1}} h(t,x,\lambda)F_i(t,x,\lambda) \rho(\lambda)\pa_{x_i}\phi(x)\, dxd\lambda dt \\
                                      & + \sum_{i,j=1}^k \int_{(t_m,\tau] \times \R^{d+1}} h(t,x,\lambda) b_{ij}(t,x, \lambda) \rho(\lambda) \pa_{x_i x_j}\phi(x)\, dx d\lambda dt \\
                                      & - \int_{(t_m,\tau] \times \R^{d+1}} \phi(x) \rho'(\lambda)\, d\gamma(t,x,\la).
             \end{split}\end{equation}
             Passing to the limit as $m\to\b$, and having in mind \eqref{wc1}, we obtain
  \begin{equation}\begin{split}\nonumber
  & I(\tau)- I(0) = \sum_{i=1}^d \int_{(0,\tau] \times \R^{d+1}} h(t,x,\lambda) F_i(t,x,\lambda) \rho(\lambda)\pa_{x_i}\phi(x) \, dx \,d\lambda \,dt \\
  & + \sum_{i,j=1}^k \int_{(0,\tau] \times \R^{d+1}}b_{ij}(t,x, \lambda) \rho(\lambda) h(t,x,\lambda) \pa_{x_i x_j}\phi(x)\, dx \,d\lambda \,dt \\
  & - \int_{(0,\tau] \times \R^{d+1}} \rho'(\lambda) \phi(x)\, d\gamma(t,x,\la) \underset{\tau\to 0}{-\!\!\!\longrightarrow} 0.
  \end{split}\end{equation}
             Thus, for all $\phi\in C_c^\b(\R^d)$, $\rho\in C^1_c(\R)$,
             $$
             \lim_{\tau\in E, \tau\to 0} \int_{\R^{d+1}} h(\tau,x,\lambda)\rho(\lambda) \phi(x)\,d\lambda\, dx = \int_{\R^{d+1}} h_0(x,\lambda)\rho(\lambda) \phi(x) \,d\la\, dx .
             $$
Having in mind that $h(\tau, \cdot)$ is bounded for almost every
$\tau \in \R$, and $C_c^\b(\R^{d+1})$ is dense in $L^1(\R^{d+1})$,
we complete the first part of the proof.

The second part of the proof is the same as the proof of \cite[Lemma
2]{vass}.
\end{proof}

\subsection{Scaling}

Denote
       $$
         \bar{x}=(x_1,...x_k, 0,...0)\in\R^d, \quad \tilde{x}=(0,...,0,x_{k-1},...,x_d)\in \R^d, \quad \bar{x}+\tilde{x}=x\in\R^d,
       $$
and change the variables in the following way,
 $t=\epsm \hat{t},$ $x_1= y_1 + \sqrt{\epsm} \hat{x}_1,$ \dots , $x_k= y_k + \sqrt{\epsm} \hat{x}_k,$
  $x_{k+1}= y_{k+1} + \epsm \hat{x}_{k+1}$, \dots , $x_d= y_d + \epsm \hat{x}_d$, i.e.
      \begin{equation}
      \label{*}
             (t,x,\la)= (\eps_m \hat t, \sqrt{\eps_m}\bar{\hat x} +{\eps_m}\tilde{\hat x}+y, \la),
      \end{equation}
      where $(\eps_m)_{m\in \N}$ is a sequence of positive numbers converging to zero and $y\in\R^d$ is a fixed vector.


        If we prove that for any $\rho\in C^1_c(\R)$,  the sequence
        \begin{equation}\label{scal}
        \int h^m(\hat t,\hat x,\lambda) \rho(\lambda)d\lambda:=\int h(\epsm \hat  t,\sqrt{\eps_m}\bar{ \hat x} +\eps_m\tilde{\hat x}+y,\lambda)\rho(\lambda)d\lambda,
        \end{equation} converges pointwise almost everywhere along a subsequence  (the
same subsequence for almost every $y\in \R^d$), we will obtain that function $h$ admits an averaged trace in the sense of Definition \ref{avsubt}.
        To this end, rewrite equation \eqref{homogena} in terms of
        variables $(\hat{t},\hat{x})$ given by \eqref{*}:

\begin{equation}\begin{split}\label{lhm}
L_{h^m}:=& \pa_{\hat t} h^m +\sum_{i=k+1}^d \pa_{\hat x_i} \left( F_i(\la) h^m \right)-
\sum_{i,j=1}^k \pa^2_{\hat x_i \hat x_j} (b_{ij}^m  h^m) =\\
& = \epsm \pa_\la  \gamma^m - \sqrt{\epsm} \sum_{i=1}^k \pa_{\hat
x_i} \left(  F_i^m h^m\right) =:\pa_\la \hat \gamma^m,
\end{split}\end{equation}
        where $b_{ij}^m (\hat t, \hat x, \la)= b_{ij}(\epsm \hat t, y+ \sqrt{\epsm} \bar{\hat x} +\epsm \tilde{\hat x},\la)$ and the same for $F_i^m$ and $\gamma^m$.

Let us remark that, considering the right hand side of \eqref{lhm} and due to assumption \eqref{assum}, we have that
for any $\rho\in C^1_c(\R)$, $\int \rho(\lambda) L_{h^m} d\lambda
\in {\cal M}(\R^d_+)$. Moreover, we have the following lemma whose
proof is the same as the one from \cite[Lemma 3.2]{pan_trace}.

\begin{lemma}\label{llum}
      After a possible extraction of a subsequence, for a.e $y\in \R^d$ and any $\rho\in C^1_c(\R)$,
      $$
        \int_{\R} \rho(\lambda) L_{h^m} (\hat t, \hat x,\lambda) d\lambda \to 0, \mbox{ as } m \to \b,  \mbox{ in }
        {\cal M}(\R^d_+) \ \ { strongly}.
      $$
\end{lemma}

\section{Averaged traces in the homogeneous case}

In this section, we consider equation \eqref{homogena}. We shall prove the following theorem.

\begin{theorem}\label{t-homog}
Assume that $h\in L^\infty(\R^d_+\!\times \!\R)$ is
a weak solution to \eqref{homogena} satisfying
\eqref{assum}. Then, there exists a function $h_0\in
L^\infty(\R^{d+1})$ such that \eqref{res_short} holds.
\end{theorem}

In order to prove the theorem, we need a corollary of the result
from \cite{LM} (see also Remark 16 and Section 5 there).

\begin{theorem}\cite[Theorem 7]{LM}
       \label{main-result}
       Denote by $P=\{\xi\in\R^{d+1} : \,
       \xi_0^2+\xi_1^4+...+\xi_k^4+\xi_{k+1}^2+..+\xi_{d}^2=1\}$ (the
       ultra-parabolic manifold).

       Assume that $h_n\rightharpoonup 0$  weakly in
       $L^\infty_{\rm loc}(\R^+\!\!\times\!\R^{d}\!\times \!\R)$, where $h_n$
       represent weak solutions to
       \begin{equation}
       \label{geneq}
              \pa_t h_n + \Div_x \left(\tilde{F}(t,x,\lambda)h_n\right)
              =\sum\limits_{s,j=1}^k \pa^2_{x_s x_j}\left(   {\tilde{b}_{sj}(t,x,\lambda)} h_n\right)+\pa_\lambda
              \gamma_n(t,x,\lambda),
        \end{equation}
where $\tilde{F}$ and $\tilde{b}=(\tilde{b}_{sj})_{s,j=1,\dots,k}$
are continuous functions, the matrix
$\tilde{b}=(\tilde{b}_{sj})_{s,j=1,\dots,k}$ is positively definite
almost everywhere, and the sequence $(\gamma_n)_n$ is strongly
precompact in $L_{\rm loc}^2(\R;W_{\rm loc}^{-1,q}(\R^d_+))$.

Assume that for every $\xi\in P$ and almost every $(t,x)\in \R^d_+$
($i$ is the imaginary unit
       below)
       \begin{equation}\label{kingnl}
              i \left(  \xi_0 + \sum\limits_{j=k+1}^d
              \tilde{F}_j(t,x,\lambda) \xi_j \right) +\sum\limits_{s,j=1}^k {\tilde{b}_{sj}(t,x,\lambda) } \xi_s
              \xi_j \not= 0 \quad {\rm a.e.} \; \lambda\in \R.
       \end{equation}

Then, for any $\rho\in C^1_c(\R)$,
       $$
       \int_{\R}h_n(t,x,\lambda)\rho(\lambda)d\lambda \to 0 \ \ \text{
       strongly in $L^1_{\rm loc}(\R^+\times\R^d)$}.
       $$
\end{theorem}

\begin{remark}
Since the matrix $b$ is positively-definite almost everywhere, it is enough to assume that for almost
       all $x\in \R^d$ and every $\tilde\xi\in \R^{d-k}$, $\tilde\xi\neq 0$,  the function $\lambda \mapsto \sum\limits_{j=k+1}^d
              \tilde{F}_j \xi_j$ is not constant  on a set of positive measure.
Remark that we can also use somewhat weaker assumptions given in \cite[Definition 2]{pan_JMS}.

Remark also that, due to linearity of equation \eqref{geneq}, the
condition $h_n\rightharpoonup 0$ in Theorem \ref{main-result} can be replaced by boundedness of
$(h_n)_{n\in\N}$. In that case, there exists a subsequence of
$(h_n)_{n\in\N}$ (not relabeled here) such that for some $h\in
L^\infty_{\rm loc}(\R_+^{d}\!\times \!\R)$, it holds
$h_n\rightharpoonup h$. Then, the sequence $(h_n-h)_{n\in\N}$
satisfies equation of type \eqref{geneq}, and we have
$$
       \int_{\R^m}h_n(t,x,\lambda)\rho(\lambda)d\lambda \to  \int_{\R^m}h(t,x,\lambda)\rho(\lambda)d\lambda \ \ \text{
       strongly in $L^1_{\rm loc}(\R^+\times\R^d)$}.
$$
\end{remark}

Next, we need the following proposition.

\begin{proposition}\label{th3mod}
         Let $h$ be a distributional solution to \eqref{homogena} satisfying \eqref{assum},
         and suppose that, in \eqref{homogena}, the component $F_d$ of the flux vector $F$ is absent,
         i.e. equation \eqref{homogena} has the form
        \begin{align}\label{1_red}
        &\pa_t h+ \sum_{i=1}^k \pa_{x_i}(F_{i}(t,x,\lambda)h)+\sum_{i=k+1}^{d-1} \pa_{x_i}(F_{i}(\lambda) h) \\
        &=\sum_{i,j=1}^k \pa_{x_i x_j}(b_{ij}(t,x, \lambda) h)+\pa_\lambda \gamma(t,x,\lambda), \ \ {\rm in} \ \ {\cal
        D}'(\R^+\!\!\times\! \R^d\!\times\! \R).  \nonumber
        \end{align}
        Then for a.e. $x_d\in\R$, $\tilde{h}(t,x',\lambda):=h(t,x',x_d,\lambda)$, $x'=(x_1,...,x_{d-1})$, is a weak solution to
        the (reduced) equation \eqref{1_red}, and $x_d$ is treated like a parameter.
\end{proposition}

\begin{proof}
Take $\psi\in L_c^\infty(\R)$, $\phi\in C^2(\R^+\times \R^{d-1})$,
      and $\rho\in C^1_c(\R)$ and test \eqref{homogena} on $\phi(t,x') \psi(x_d) \rho(\lambda)$ to obtain
\begin{equation} 
      \Big| \underset{\R^+\times\R^d\times \R}{\int} \psi(x_d)\phi(t,x')\rho'(\lambda)d \gamma(t,x',x_d,\lambda)\Big| \leq
      C(\|\rho\|_{C^1},\|\phi\|_{C^2})\|\psi\|_{L^\infty}.
\end{equation}
Furthermore, it is not difficult to see that
$\int_{\R^+\times\R^{d-1}\times \R}\phi(t,x')\rho'(\lambda)d
\gamma(t,x',x_d,\lambda)\in {\cal M}(\R_d)$ is absolutely continuous
with respect to the Lebesgue measure. Thus, for every fixed $\rho$
and $\phi$, the function $x_d\mapsto \int_{\R^+\times\R^{d-1}\times
\R}\phi(t,x')\rho'(\lambda)d \gamma(t,x',x_d,\lambda)$ belongs to
$L^1_{\rm loc}(\R)$. Therefore, there exists a set $E$ of full
measure such that \eqref{1_red} is satisfied for $x_d\in E$ for the
taken functions $\rho$ and $\phi$.

Now, take test functions $\phi\rho\in S$, $\phi=\phi(t,x')$,
$\rho=\rho(\la)$,  where $S$ is a countable dense subset of
$C^2_c(\R_+^{d-1})\times C^1_c(\R)$. Again, we can find a set $E$ of
full measure such that \eqref{1_red} is satisfied for $x_d\in E$ and
all  $\phi\rho\in S$. Using the density of $S$ in
$C^2_c(\R_+^{d-1})\times C^1_c(\R)$, we conclude that for every
$x_d\in E$, \eqref{1_red} holds for any $\eta=\phi\rho\in
C^2_c(\R_+^{d-1})\times C^1_c(\R)$, i.e. due to the density
arguments again, for any $\eta\in C^2(\R^{d-1}_+\times \R)$.
\end{proof}

$\;$

{\bf Proof of Theorem \ref{t-homog}:} First assume that equation
\eqref{homogena} is such that for almost every $y\in \R^d$ the flux
and the diffusion matrix
\begin{equation}
\label{new_flux} (1,F(0,y,\lambda)) \ \ {\rm and} \ \
b(0,y,\lambda)=(b_{sj}(0,y,\lambda))_{s,j=1,\dots k}
\end{equation}respectively, are non-degenerate on an interval $(a,b)$ in the
sense of \eqref{kingnl} (we replace there $\tilde{F}$ and
$\tilde{b}$ by $F(0,y,\lambda)$ and $b(0,y,\lambda)$). Consider the
sequence $h^m(\hat t, \hat x, y,\la)=h(\eps_m \hat t, \sqrt{\eps_m}
\bar{\hat x}+\eps_m \tilde{\hat x}+y,\lambda)$ defined in
\eqref{scal}, that satisfies \eqref{lhm}. For a fixed $y$,  the
sequence  $(h^m)_m$ is bounded in $L^\b(\R^{d+1}_+)$ and has a
weakly-$\star$ convergent subsequence (not relabeled). Denote by
$\tilde{h}_0(\hat t,\hat x,y,\lambda)$ its weak-$\star$ limit.

Notice that the sequence of functions $w_m=h^m-\tilde{h}_0$
fulfills assumptions of Theorem \ref{main-result} (i.e. it satisfies
equation \eqref{geneq} with $\tilde{F}=F(0,y,\cdot)$ and
$\tilde{b}=b(0,y,\cdot)$), so there exists a subsequence of
$(h^m)_m$ (not relabeled) such that for any $\rho \in C^1_c(a,b)$
\begin{equation}
\label{nv1} \int_{a}^b h^m(\hat t,\hat x,y,\lambda) \rho(\lambda)d\lambda \to
\int_{a}^b\tilde{h}_0(\hat t, \hat x,y,\lambda)\rho(\lambda)d\lambda \ \ {\rm
in } \ \ L^1_{\rm loc}(\R^d_+\times \R).
\end{equation}
According to Lemma \ref{llum}, the function $\tilde{h}_0$ satisfies (we remind that $y$ is fixed):
$$
\pa_{\hat t} \tilde{h}_0 + \sum_{i=k+1}^d \pa_{\hat x_i} \left(F_i(\la) \tilde h_0\right)
=\sum\limits_{i,j=1}^k
             \pa^2_{\hat x_i  \hat x_j}\left(   { \hat b_{ij}(0,y,\lambda)  }
             \tilde{h}_0\right).
$$
Since the last equation is linear, the function $\tilde{h}_0$ is
also its isentropic solution (see e.g. \cite{pan_kand}) and thus,
repeating the proof of \cite[Proposition 2]{pan_trace}, we conclude
that $\tilde{h}_0$ admits the {strong} trace at $t=0$. From here and
according to {Proposition \ref{wt}}, we conclude that
$$
\tilde{h}_0(0,\hat x,y,\lambda)=h_0(y,\lambda),
$$ where $h_0$ is defined in Proposition \ref{wt}. Details can be found in the proof of \cite[Theorem 2]{pan_trace}.

Since a solution to a Cauchy problem for linear ultra-parabolic
equations with regular coefficients must be unique, we conclude
that, for almost every $y\in \R^d$, the limit from \eqref{nv1} does
not depend on the choice of the subsequence. Moreover, we conclude
that for almost every $(y,\lambda)\in \R^d\times \R$, it holds
$$
\tilde{h}_0(\hat t,\hat x,y,\lambda)=h_0(y,\lambda).
$$
Thus, for $t_m=\eps_m \hat t$ ($\hat t>0$ is fixed)
$$
\int_{a}^b h(t_m,\hat x, y,\lambda)\rho(\lambda)d\lambda \to \int_{a}^b
h_0(y,\lambda)\rho(\lambda)d\lambda \ \ {\rm in} \ \
L^1_{\rm loc}(\R^d).
$$
Since the sequence $(t_m)$ is arbitrary, this proves Theorem
\ref{t-homog} in the non-degenerate case.\\

Now assume that  the non-degeneracy is lost in $(a,b)$ for the flux
and the diffusion from \eqref{new_flux}. In this case we use the
method of mathematical induction with respect to
$d-k\geq 0$.\\

\noindent{\em Step 1.}
                Assume that $d-k=0$. In this case, equation
                \eqref{homogena} reduces to the strictly parabolic
                equation which means that it satisfies the
                non-degeneracy condition from \eqref{kingnl} for the flux and the diffusion from \eqref{new_flux}. Thus, the existence of traces follows
                from the first part of the proof.\\

\noindent{\em Step 2.}
             Assume that if $h\in L^\b(\R^{d-1}_+\times \R)$, $h=h(t,x_1,...,x_{d-1},\lambda)$, is a weak solution to \eqref{1_red},
             then  there exists $h_0\in L^\b (\R^{d-1}\times \R)$ and a set $E\subset \R^+$ of full measure, such that for all zero sequences $t_m\in E$ and every $\rho\in C^1_c(\R)$
             $L^1_{\rm loc}(\R^{d-1})-\lim_{m\to \infty} \int_{\R} h(t_m,x,\lambda)
             \rho(\lambda) d\lambda = \int_{\R} h_0(x,\lambda)
             \rho(\lambda) d\lambda $.\\

\noindent{\em Step 3.}
         Let $h\in L^\b(\R^{d}_+\times \R)$ be a weak
         solution to \eqref{homogena}.
         Since the non-degeneracy is lost in $(a,b)$,  there exists
         nonzero vector $(\xi_{k+1},...,\xi_d)\in \R^{d-k+1}$ and a constant $c_d$ such
         that
         \begin{equation}\label{gnc}
                 \xi_{k+1} F_{k+1}(\la)+...+\xi_d  F_d(\la)=-c_d, \ \ \lambda\in (a,b).
         \end{equation}
         Introduce the change of spatial variables
         $\tilde x\equiv(x_{k+1},...,x_d)\in \R^{d-k}\mapsto (z_{k+1},...,z_d)$ $\equiv\tilde z\in \R^{d-k}$ as
         $\tilde{z}=ct+A\tilde{x}$, where $c=(c_{k+1},...,c_d)^\top$ and
         $A=[a_{ij}]_{i,j=k+1,...,d}\in\R^{d-k\times d-k}$, $a_{ij}=a_{ji}$.
         Other spatial variables will remain unchanged, i.e. $z_1=x_1$, ..., $z_k=x_k$.
         With this change, for $h=h(t,z,\lambda)$, equation \eqref{homogena} becomes
        $$
          h_t + \sum_{i=1}^{k}\pa_{z_i} \left[ F_{i}(t,z,\lambda)h \right]+
          \sum_{l=k+1}^{d} \pa_{z_l}\left[ \left( c_l  + \sum_{i=k+1}^{d}a_{li} F_{i} \right) h \right]=
          \sum_{i,j=1}^{k}\pa_{z_j} \left[ b_{ij} \pa_{z_i}h\right]+\pa_\la \gamma.
        $$
        Denote $\tilde{F}_{l} := c_l  + \sum_{i=k+1}^{d}a_{li} F_{i} $, $l=k+1,...,d$, and $\tilde{F}_{i}:=F_{i}$, $i=1,...,k$.
        According to \eqref{gnc}, we choose 
        $a_{d,k+1}:=\xi_{k+1}$, ..., $a_{d,d}:=\xi_{d}$ and obtain
        $\pa_{z_d}\tilde{F}_d(t,z,\lambda) 
        =0$, for $\lambda\in(a,b)$.
        Thus, for $\lambda\in (a,b)$, the equation takes the following form,
        \begin{align}\label{induk}
               \pa_t h+\sum\limits_{i=1}^{k}\pa_{z_i}(F_{i}(t,z,\lambda)
               h)+\sum\limits_{i=k+1}^{d-1} \pa_{z_i}
               (\tilde{F}_i(\lambda)h)=\sum_{i,j=1}^k \pa^2_{z_iz_j}
               (b_{ij} h)+\pa_\la \gamma.
        \end{align}
        According to Proposition \ref{th3mod}, for a fixed (parameter)
        $z_d$, the function $h=h(t,z',z_d,$ $\lambda)$, $z'\in\R^{d-1}$ is a
        weak solution to \eqref{induk}, i.e. $h(t,z',z_d,\lambda)$ is a weak
        solution to \eqref{induk}, for a.e. $z_d\in\R$.

        Then, according to the inductive hypothesis, for a.e. $z_d\in\R$,
        there is $\tilde{h}_0(z',\la)\equiv \tilde{h}_0[z_d]\in L^\b(\R^{d-1}\times \R)$ such that  for  all $\rho\in C^1_c(\R)$,
        $$
        L^1_{\rm loc}(\R^{d-1})-\lim_{m\to \infty} \int_{\R} h(t_m,z',z_d,\lambda)
             \rho(\lambda) d\lambda = \int_{\R} h_0[z_d](z',\lambda)\rho(\lambda) d\lambda,
        $$
        for a sequence $(t_m)$ of positive numbers tending to zero.

        To obtain the analogical assertion in $\R^d$, we need a special choice of $(t_m,z_d)$, so we use the following
        construction (the same construction is used in \cite{pan_trace}).
        Denote

                \begin{align*}
E:=  &  \{t>0\,|\, (t,x,\lambda) \mbox{ is a Lebesgue point to }
h(t,x,\lambda) \mbox{ for a.e. } (x,\lambda)\in\R^d\times\R
        \},\\
               \cal M:=  & \{(t,z,\lambda)\equiv (t, z',z_d,\lambda)\,|\, (t,z,\lambda) \mbox{ is a Lebesgue point to } h \mbox{ and }\\
                                     & \qquad \;\qquad \qquad\qquad\qquad\;(t,z',\lambda) \mbox{ is a Lebesgue point to } h(\cdot, z_d,\cdot) \},\\
               \cal M_t:= & \{(z,\la)\,|\, (t,z,\lambda)\in \cal M \},
        \end{align*}
         which are sets of full measure. There exists a subsequence $(t_{r})_{r}$ from $E'=\{t>0: \; {\cal M}_t \ \ \text{is a set of full measure} \}$ such that
                $(t_r,z',\lambda)$  is a Lebesgue point to $h(\cdot, z_d,\cdot)$. Take
         $$
             z_d\in \cal Z= \bigcap_r \cal Z_r, \mbox{ where }  \cal Z_r:=\{s\in\R \, |\, (z', s,\lambda)\in \cal M_{t_r}\}.
         $$
           Applying the inductional hypothesis to $h(t_r,z',z_d,\lambda)$ we obtain
           that there exists $\tilde{h}_0(\cdot,z_d,\la)\in L^\b(\R^{d-1}\times \R)$ such that
           $$
           L^1_{\rm loc}(\R^{d-1})-\lim_{r\to \infty} \int_{\R} h(t_r,z',z_d,\lambda)
             \rho(\lambda) d\lambda = \int_{\R} \tilde{h}_0(z',z_d,\lambda)
             \rho(\lambda) d\lambda.
           $$
           With the choice $z=(z',z_d)$, we have that $\tilde h_0(z,\la)= \tilde h_0 (z',z_d,\la)\in L^\b(\R^{d}\times \R)$,
           and then apply the Lebesgue dominated convergence theorem to conclude that
           $$
             L^1_{\rm loc}(\R^{d})-\lim_{r\to\b} \int_{\R}h(t_r,
             z,\lambda)\rho(\lambda)d\lambda= \int_{\R} \tilde h_0 (z,\la)\rho(\lambda)d\lambda.
           $$
           Now, the same limit relation follows for the original variable $x$, i.e.
           $$
             L^1_{\rm loc}(\R^{d})-\lim_{r\to\b}\int_{\R}h(t_r,x,\lambda)\rho(\lambda)d\lambda= \int_{\R} \tilde h_0 (x,\la)\rho(\lambda)d\lambda.
           $$
           $\Box$





\section{The heterogeneous case; proof of the main theorem}

In this section we shall prove the main result of the paper. First,
we introduce the notion of traceability.
%
%

\begin{definition}
\label{d-assump} We say that the coefficients of equation
\eqref{opsta} satisfy {\em the strong traceability conditions} at
the point $(x^0,\lambda^0)\in \R^d\times \R$ if one of the following
two conditions are satisfied:

\vspace{0.5cm}

{\em 1)} There exists a neighborhood of $(0,x^0,\lambda^0)$ in which
the non-degeneracy condition \eqref{kingnl} is fulfilled for the
flux and the diffusion from \eqref{new_flux}, or

\vspace{0.5cm}

{\em 2)} There exists a neighborhood $U$ of $x^0 \in \R^d$, $T>0$,
an interval $\lambda_0 \in (\alpha,\beta)\subset \R$, and a regular
change of variables $(\hat{t},\hat{x}): (0,T)\times U \to
(0,\hat{T})\times \hat{U} \subset \R^{d+1}$ such that

\begin{itemize}

\item For all $s=k+1,..., d$, $t\in (0,T)$, $x\in U$ and $\lambda \in (\alpha,\beta)$
\be\label{kapaF}
\hat F_s             
+\sum\limits_{i,j=1}^k b_{ij} \cdot
\left(2 \sum\limits_{r=0}^d\pa_{\hat{x}_r} \left[ \frac{\pa
\hat{x}_r}{\pa x_j}     \frac{\pa \hat{x}_s}{\pa
x_i} \right] - \frac{\pa^2 \hat x_s}{\pa x_i \pa x_j} )\right)=:p_s(\lambda)
\ee
is independent of $x\in U$, $t\in(0,T)$,  
where we denote $\hat F_s = \ds \sum\limits_{i=0}^d F_i \frac{\pa \hat{x}_s}{\pa x_i}$, for $s=0,...,d$,  $F_0=1$, $x_0=t$ and $\hat x_0=\hat t$.

\item The coefficient
\begin{equation}
\label{t-coord} \hat{F}_0=\sum\limits_{j=0}^d F_j \frac{\pa
\hat{t}}{\pa x_j} \equiv const \neq 0.
\end{equation}

\item The matrix
\be\label{kapab}
\left( \sum\limits_{i,j=1}^k  b_{ij} \frac{\pa \hat{x}_s}{\pa x_i} \frac{\pa \hat{x}_r}{\pa x_j} \right)_{s,r=0,\dots,d}=: (\hat{b}_{sr})_{s,r=0,\dots,d},
\ee
has the same properties as the matrix $b$ from \eqref{3} but with respect to the variables
$(\hat{x}_0,\hat x_1,..., \hat x_d)$.
\end{itemize}
\end{definition}

\begin{remark}
The matrix $\hat b$ defined in \eqref{kapab} should have null entries for $\max \{s,r\}>k$ and also for $\min \{s,r\}=0$. This implies that
$$
\frac{\pa \hat{x}_s}{\pa x_i} = 0 \quad  \mbox{ if } \quad s\in\{0,k+1,k+2,...,d\} \mbox{ and } i\in \{1,...,k\},
$$ i.e. that the variables $\hat{x}_s= \hat{x}_s(t, x_{k+1},..., x_d)$,
$s\in\{0,k+1,k+2,...,d\} $,   do not depend on $x_i$, $i=1,...,k$. This
fact reduces  condition \eqref{kapaF} to
$$
\hat F_s   = \frac{\pa \hat{x}_s}{\pa t} +   \sum\limits_{i=k+1}^d F_i \frac{\pa \hat{x}_s}{\pa x_i}
=p_s(\lambda), \quad s=k+1,...,d
$$
and condition
\eqref{t-coord} to
$$
\hat{F}_0=\frac{\pa \hat{t}}{\pa
t}+\sum\limits_{j=k+1}^d F_j \frac{\pa \hat{t}}{\pa x_j} \equiv
const.
$$

We would like to thank to the referee for this remark.
\end{remark}


\begin{theorem}\label{mainalbas}
       Let  $h\in L^\infty(\R^d_+\!\times \!\R)$
       be a weak solution to \eqref{opsta}. Moreover, assume that equation \eqref{opsta}
       satisfies the strong traceability conditions from Definition \ref{d-assump} on a set $E\times F \subset \R^d\times \R$ of full measure.
       Then the function $h$ admits an averaged trace in the sense of Definition \ref{avsubt}.
\end{theorem}

\begin{proof}
Fix $(x^0,\lambda^0)\in \R^{d+1}$ and the neighborhood $U\times
(\alpha,\beta)$, $U= U(x^0)$, so that the strong traceability
property is fulfilled. If the flux $(1, F(0,x,\la))$ and diffusion
matrix $b(0,x,\la)$ are non-degenerate in that neighborhood, then we
repeat the first part of the proof of Theorem \ref{t-homog} to
conclude that the function $h$ admits an averaged trace on $U\times
(\alpha,\beta)$.

Consider now the change of variables from  the second condition from the strong traceability
definition. Using the same notation as in Definition \ref{d-assump} we calculate
\begin{align}
\label{tr1}
 &  
\sum_{i=0}^d \pa_{x_i} (F_i h) =  \sum_{i=0}^d \sum_{s=0}^d \frac{\pa(F_i h)}{\pa \hat x_s} \frac{\pa \hat x_s}{\pa {x_i}} =\\
& = 
\sum_{s=0}^d \pa_{\hat x_s} [\hat F_s h]
- h  \sum_{i,s=0}^d F_i \, \pa_{\hat x_s} \left[  \frac{\pa \hat x_s}{\pa {x_i}} \right]\nonumber,
\end{align}
where $\hat F=(\hat F_0,...,\hat F_d)$, and


\begin{align}
 \label{tr2}
&\sum_{i,j=1}^k \pa^2_{x_i x_j} (b_{ij} h)
=\sum_{i,j=1}^k \left(  \sum_{s,r=0}^d  \frac{\pa^2(b_{ij} h)}{\pa \hat x_s \pa \hat x_r } \frac{\pa \hat x_s}{\pa {x_i}} \frac{\pa \hat x_r}{\pa {x_j}}
+ \sum_{s=0}^d \frac{\pa(b_{ij} h)}{\pa \hat x_s } \frac{\pa^2 \hat x_s}{\pa {x_i} \pa x_j}
 \right)=\\
\nonumber
&=  \sum_{s,r=0}^d \pa^2_{ \hat x_s \hat x_r} [\hat b_{sr} h ] + \sum_{s=0}^d \pa_{ \hat x_s }\left[
h\sum_{i,j=1}^k b_{ij} \cdot\left(
\frac{\pa^2 \hat x_s}{\pa {x_i} \pa x_j} -
2\sum_{r=0}^d \pa_{\hat x_r}\left[
 \frac{\pa \hat x_s}{\pa {x_i}} \frac{\pa \hat x_r}{\pa {x_j}}
\right]
\right)
 \right]\\
 \nonumber
 & +h \sum_{i,j=1}^k  b_{ij}  \left( \sum_{s,r=0}^d
  \pa^2_{ \hat x_s \hat x_r}
 \left[
 \frac{\pa \hat x_s}{\pa {x_i}} \frac{\pa \hat x_r}{\pa {x_j}}
 \right]
 - \sum_{s=0}^d
 \pa_{\hat x_s} \left[
 \frac{\pa^2 \hat x_s}{\pa {x_i} \pa x_j}
 \right]
 \right)
\end{align}
By inserting \eqref{tr1} and \eqref{tr2} into \eqref{opsta}, we get from the strong traceability
conditions (keep in mind the fact that the matrix $B$ is symmetric)
\begin{align}
\label{reduced}
&\pa_{\hat t} [\hat F_0 h]
+ \sum\limits_{s=1}^k \pa_{\hat{x}_s}
\left[
h \cdot
\left(
\hat{F}_s -\sum_{i,j=1}^{k} b_{ij}
\left(
2\sum_{r=1}^k \pa_{\hat{x}_r}
\left[ \frac{\pa \hat x_s}{\pa {x_i}} \frac{\pa \hat x_r}{\pa {x_j}} \right]
- \frac{\pa^2 \hat x_s}{\pa {x_i} \pa {x_j}}
\right)
\right)
\right]\\
& +\sum\limits_{s=k+1}^d
\pa_{\hat{x}_s} [h \cdot p_s(\lambda)]
-\sum\limits_{s,r=1}^k \pa^2_{\hat{x}_s \hat{x}_r} [h\cdot \hat{b}_{sr}] =h \cdot g(\hat t, \hat{x},\lambda) +\pa_\lambda \gamma(\hat t, \hat x,\la),  \nonumber
\end{align}
for $(\hat t, \hat x,\lambda) \in (0, \hat T)\times\hat U \times (\alpha,\beta)$, and
appropriate function $\hat{g}$.
Equation \eqref{reduced} has the form of equation \eqref{homogena}
and thus, according to Theorem \ref{t-homog}, $\hat{h}:=h(\hat t,
\hat x, \la)$ admits averaged traces on $\hat U\times
(\alpha,\beta)$. Since the transformation $(\hat t, \hat{x})$ is
regular, we conclude that the averaged traces exist on $U\times
(\alpha,\beta)$ as well.

Since the averaged traces exist in a neighborhood of almost every
point $(x_0,\lambda_0)$, by a diagonalization procedure, we conclude
that they exist globally.  
\end{proof}


\begin{remark}

Notice that from the proof of Theorem \ref{t-homog} it
follows that it is enough to demand here the strong traceability
conditions for the flux and diffusion of the form
$$
(1,F(0,x,\lambda)) \ \ {\rm and} \ \
(b_{sj}(0,x,\lambda))_{s,j=1,\dots k}
$$ which implies that in the change of variables given 
in the strong
traceability conditions we can take $\hat{t}=t$. This implies that
condition \eqref{t-coord} can be omitted.

Again, we would like to thank to the referee for this remark.
\end{remark}

The idea of reduction of the space dimension and the
previous theorem can be used to weaken the strong traceability conditions.

\begin{definition}
\label{d11} We say that the coefficients of equation \eqref{opsta}
satisfy {weak traceability conditions} at the point
$(x^0,\lambda^0)\in \R^d\times \R$ if there exists a neighborhood
$(0,T)\times U\times (\alpha,\beta)$ of the point
$(0,x^0,\lambda^0)$, and a regular transformation $\hat{x}:
(0,T)\times U\to (0,\hat T)\times \hat{U}$ which reduces equation
\eqref{opsta} on
\begin{align*}
\pa_{\hat{t}} \hat{h} &+ \sum\limits_{i=l}^d
\pa_{\hat{x}_i}(\hat{F}_i(\hat{t},\hat{x},\lambda)\hat{h})
=\sum\limits_{i,j=1}^k \pa^2_{\hat{x}_i
\hat{x}_j}\left(\hat{b}_{ij}(\hat t,\hat x,\lambda) \hat{h}
\right)\\&+\sum\limits_{m=1}^k\pa_{x_m}(\hat{F}(\hat{t},\hat{x},\lambda)\hat{h})
+ \hat{g}(\hat{t},\hat{x},\lambda)\hat{h}+\pa_\lambda
 \hat{\gamma}(\hat{t},\hat{x},\lambda),
\end{align*} where $ k< l\leq d$, and the last equation satisfies the strong
traceability conditions on $(0,\hat T)\times \hat{U} \times
(\alpha,\beta)$ with $\hat{x}_{k+1},\dots,\hat{x}_{l-1}$ considered
as parameters. The functions $\hat{F}$, $\hat{b}$, and $\hat{g}$ are
computed as in the proof of Theorem \ref{mainalbas}, while
$\hat{\gamma}$ equals to $\gamma((t,x)(\hat{t},\hat{x}),\lambda)$.
\end{definition}


The following theorem is a consequence of the induction arguments
and Proposition \ref{th3mod}.

\begin{theorem}\label{mainalbas2}
       Let  $h\in L^\infty(\R^+\!\!\times \!\R^d\!\times \!\R)$
       be a weak solution to \eqref{opsta}. Moreover, assume that equation \eqref{opsta}
       satisfies the weak traceability conditions from Definition \ref{d11}. Then, there exists a function $h_0\in
       L^\infty(\R^d\times \R)$ such that \eqref{res_short} holds.
\end{theorem}

\section{Nonlinear ultra-parabolic equations}
In this section, we shall consider the following equation
\begin{equation}\label{opsta1}
     \pa_t u+\Div_x f(t,x,u)-\sum_{i,j=1}^k \pa^2_{x_i x_j}(B_{ij}(t,x,u))=0,
\end{equation}
where $-M\leq u \leq M$, $f\in C^1(\R^d_+\times\R;\R^d)$, $F=\pa_u
f\in C^1(\R^{d+1}_+)$ and $b_{ij}= \pa_u B_{ij} \in
C^1(\R_+^{d+1})$, $i,j=1,...,k$ satisfy \eqref{3}-\eqref{2h}. The
latter equation admits several weak solutions and it is usual to
impose a conditions that solutions must satisfy in order to be
physically relevant. Such conditions are called the entropy
admissibility conditions \cite{CK,PC,pan_JMS}, and they represent an
extension of the Kruzhkov admissibility concept for scalar
conservation laws \cite{kru}, cf. \cite{CK,PC,pan_JMS}. Their
natural generalization is quasi-solutions concept introduced in
\cite{pan_trace} for scalar conservation laws. We shall adapt it
here to our situation. Before that, remark that since $f\in
C^2(\R^d\times\R;\R^d)$, the first order derivatives of $f$ are
locally bounded and we have that


       $$
          \sum\limits_{i=1}^d D_{x_i}f_i(t,x,c)=\gamma_c\in{\cal M} (\R^d), \quad c\in\R,
       $$
       where ${\cal M}$ stands for locally finite Borel measures.


\begin{definition}
We say that $u\in L^\infty(\R^d_+)$ is a {quasi solution} to
\eqref{opsta1} if it satisfies
\begin{equation}
\label{cond_1} \sum\limits_{i=1}^k \sigma ^b_{li}(u)\pa_{x_i} u \in
L^2(\R^d_+), \ \ l=1,\dots,k,
\end{equation} and for almost every $c\in \R$, $(t,x)\in\R^d_+$, there exists a
locally finite Borel measure $\mu_c\in {{\cal M}}(\R^d_+)$, such
that
             \begin{align}
             \label{cond_3}
                   L_u^c(t,x)& \equiv  \pa_t|u-c|+\Div \left[{\rm sgn}(u-c)
                    (f(t,x,u)-f(t,x,c))\right]\\
                    &-\sum\limits_{i,j=1}^k \pa^2_{x_i x_j}\left({\rm sgn}(u-c)(B_{ij}(t,x,u)- B_{ij}(t,x,c) ) \right)
                    =-\mu_c(t,x), \nonumber
             \end{align}
             in ${\cal D}'(\R^d_+)$. The family of measures $\mu_c$, $c\in\R$, is called the entropy defect-measure corresponding to $u$.
\end{definition}

Putting $c>\|u\|_\b$ in \eqref{cond_3}, it follows that there exists
$\mu\in {{\cal M}}({\R^d_+})$, such that
      \begin{equation}\label{cond_31}
                 \pa_t u+\Div_x f(t,x,u)-\sum_{i,j=1}^k \pa^2_{x_i x_j}(B_{ij}(t,x,u))=-\mu,
      \end{equation}
       in ${\cal D}'(\R^d_+).$

\begin{lemma}\label{sab}
       If $u$ is a quasi  solution to \eqref{opsta1}, then
       the cut-off function
       $$
         s_{a,b}(u)(t,x)=\max \{a, \min\{u(t,x), b\}\},
       $$
       for $a,b \in \R$, $a<b$, is a quasi  solution to \eqref{opsta1} as well.
\end{lemma}

\begin{proof}
Denote, $v=v(t,x)=s_{a,b}(u(t,x))$ and $c'=\max \{a, \min \{c, b\}\}$. Notice that
         \begin{equation}
\label{cutoff}
           s'_{a,b}(\lambda)=\begin{cases} 1, & a<\lambda <b\\
           0, & \la<a \mbox{ or } \la>b
           \end{cases}.
         \end{equation}
For a continuous function $F=F(t,x,\la)$,  denote $S(v,c)= \sgn(v-c) (F(t,x,v)-F(t,x,c))$. One can verify that

\begin{eqnarray*}
S(v,c)  & =\begin{cases}
S(u, c') - \frac 12 \Big( S(u,a) + S(u,b) \Big)  + \frac 12 \Big( F(a)+F(b) \Big)- F(c), & c<a\\
S(u, c') - \frac 12 \Big( S(u,a) + S(u,b) \Big) + \frac 12 \Big( F(b)-F(a)\Big) , & a<c<b\\
S(u, c') - \frac 12 \Big( S(u,a) + S(u,b) \Big)  - \frac 12 \Big( F(a)+F(b) \Big)  +F(c), & b<c
                        \end{cases}\\
                        &= S(u, c') - \frac 12 \Big( S(u,a) + S(u,b) \Big)  + \frac 12 \Big( S(a,c)+S(b,c) \Big) \qquad \qquad \qquad \qquad
\end{eqnarray*} The same arguing holds for $\sgn(v-c)
(b(t,x,v)-b(t,x,c))$. This enables us to conclude that
\begin{eqnarray*}
               L_v^c & =-\mu_{c'}+\frac12\Big(\mu_a+\mu_b +
               \sgn (a-c)( \gamma_a-\gamma_c) + \sgn (b-c) (\gamma_b-\gamma_a) \Big),\\
               & =   \begin{cases}
-\mu_{c'}+\frac12(\mu_a+\mu_b  +\gamma_b+\gamma_a) - \gamma_c,  & c<a\\
-\mu_{c'}+\frac12(\mu_a+\mu_b  +\gamma_b-\gamma_a), & a<c<b\\
-\mu_{c'}+\frac12(\mu_a+\mu_b  -\gamma_b-\gamma_a) + \gamma_c, & b<c
                        \end{cases}     \qquad          \qquad      \qquad      \quad
\end{eqnarray*}
which proves that $v$ is a quasi  solution to \eqref{homogena}.
\end{proof}

A simple consequence of \eqref{3} and \eqref{cond_1} is the
following lemma.

\begin{lemma} \label{regul}
      Let $u$ be an quasi solution to \eqref{opsta1}.
      Then, for every $j\in \N$ and any $a<b$ such that $(a,b)\subset  (\lambda_j,\lambda_{j+1})$, $m\in \N$, ($\lambda_j$ are given in \eqref{4})
      $$
          \pa_{x_i}s_{a,b}(u) \in L^2_{\rm loc}(\R_d^+), \quad i=1,\dots, k.
      $$
\end{lemma}
\begin{proof}

         From \eqref{3}, \eqref{cond_1}, and \eqref{cutoff}, we conclude,
         \begin{align*}
                  \sum\limits_{i=1}^k|\pa_{x_i}s_{a,b}(u)|^2 & \leq
                  \frac{s'_{a,b}(u)}{c(u)}\sum\limits_{i,j=1}^k b_{ij}(t,x,u) u_{x_i}
                  u_{x_j}\\
                  & \leq \max_{a\leq \lambda \leq b}(c(\lambda))^{-1}\sum\limits_{j=1}^k\left(\sum\limits_{i=1}^k \sigma_{ij}\pa_{x_i}
                  s_{a,b}(u)
                  \right)^2\in L^1_{\rm loc}(\R^+_d),
         \end{align*}
         where the last relation follows from \eqref{3}.  This concludes the proof.
   \end{proof}

$$\;$$

We shall need the following characterization of entropy solutions to
\eqref{opsta1} (see also \cite{Dal}).

\begin{proposition}\cite{PC}\label{prop-kin}
        The function $u$ represents a quasi  solution to
        \eqref{opsta1} if and only if the kinetic function
        \begin{equation}
        \label{kinh}
        h(t,x,\lambda)=\begin{cases} 1, & 0\leq \lambda \leq u(t,x)\\
                                    -1, & u(t,x) \leq \lambda \leq 0\\
                                     0, & otherwise
                        \end{cases}
        \end{equation}
        satisfies the following linear equation,
        \begin{equation}\label{kin_1h}\begin{split}
        &\pa_th(t,x,\lambda)+\Div \left[\pa_\lambda f(t,x,\lambda) h(t,x,\lambda)\right]\\
                    & -\sum\limits_{i,j=1}^k \pa^2_{x_i x_j}  \left( b_{ij}(t,x,\lambda)  h(t,x,\lambda) \right) =-\pa_\lambda m(t,x,\lambda),
               \quad \mbox{ in } {\cal D}'(\R^d_+\times\R),
        \end{split}\end{equation}
        where $m\in {{\cal M}}(\R^d_+\times \R)$ and $b_{ij}=\pa_\lambda B_{ij}$.
\end{proposition}

Now, we can state the main result of this section:
\begin{theorem}\label{mainalbas1} Assume that $u$ is a quasi solution to
\eqref{opsta1} and that the coefficients of equation \eqref{kin_1h}
satisfy the weak traceability property. Then, there exists a
function $u_0\in L^\b(\R^d)$, such that
         $$
         L^1_{\rm loc}(\R^d)-\lim_{t\to 0} u(t,\cdot) = u_0.
         $$
\end{theorem}

{\bf Proof:}        Let $F$ be dense countable subset of $\R$, such that for all $c\in F$, \eqref{cond_3} holds.
       Remark that for any $g\in L^\infty(\R_+^d\times \R) $ the function $h$ defined in \eqref{kinh} satisfies
       \begin{equation}
             \int_{\R}\int_{\R_+^d} h(t,x,\lambda) g(t,x,\lambda)dt dx
             d\lambda=\int G(t,x,u(t,x))dt dx,
       \end{equation}
       where $G(t,x,v)=\int_0^v g(t,x,\lambda)d\lambda$, (see e.g. \cite[(1.2.5)]{brenier}). If we take $g(t,x,\lambda)={\rm sgn}(\lambda-c)$,
       from Proposition \ref{wt} we have that there exists a function $u_0 \in L^\b(\R^d)$, and functions $v_c\in L^\b(\R^d)$
such that for every $c\in \R$
       \be\label{ccc}
             u(t, \cdot)\, -\!\!\!\rightharpoonup u_0 \quad \mbox{ and } \quad |u(t,\cdot)-c| -\!\!\!\rightharpoonup v_c,
       \ee
       weakly - $\star$ in $L^\b(\R^d)$, as $t\to 0$, $t\in E$.

       Furthermore, according to Theorem \ref{mainalbas2}, there exists the
averaged trace for the function $h$ defined in \eqref{kinh}. If we
notice that
$$
\int h(t,x,\lambda)\rho(\lambda)d\lambda =u,
$$ if $\rho\equiv 1$ on $(-M,M)$, we conclude that $u(t,\cdot) \to u_0$ in $L^1_{\rm
       loc}(\R^d)$. $\Box$

\begin{example}
 Consider the one-dimensional scalar conservation law,
         \begin{equation*}
                   \pa_t u +\pa_x (x u)=0.
         \end{equation*}
         In this case, we simply take $\hat{x}(x)={\rm
         ln}|x|$, $x\neq 0$, and we infer that the strong
         traceability conditions are fulfilled locally almost
         everywhere (i.e. on the set $(-\infty,0)\cup (0,\infty) $) implying that the traces at $t=0$ exist.

Less trivial example is the two dimensional scalar conservation law
          which is linear in the direction of the first space variable (i.e.
          it is not non-degenerate),
          \begin{equation*}
                  \pa_t u +\pa_{x_1} (x_1 u)+\pa_{x_2}(x_1 u^2)=0.
          \end{equation*}
          In this case, we first choose $\hat t=t$, $\hat{x}_1 ={\rm
          ln}|x_1|$, $x_1\neq 0$,  $\hat{x}_2 =x_2$. Locally, this reduces equation
          to
          $$
          \pa_{\hat t} \hat{u} +\pa_{\hat{x}_1}\hat{u} +\pa_{\hat x_2}(e^{\hat{x}_1}
          \hat{u}^2)=\hat{u},
          $$
          where $\hat u =u(\hat t, \hat x)$. Then, we take
          $\tilde{t} = \hat x_1+ \hat t$, $z_1=\hat x_1 - \hat t$, $z_2=\hat x_2$, and
          we thus reduce the equation to
          $$
          2 \pa_{\tilde{t}} \tilde{u} +\pa_{z_2}(e^{\frac{1}{2} ( {\tilde{t}+z_1} )}
          \tilde{u}^2)=\tilde{u}.
          $$
          This equation satisfies condition {\em 1)} from Definition 8 if we consider
          $z_ 1$ as a parameter (and, accordingly, its kinetic counterpart is also non-degenerate). Thus, the weak traceability
          conditions are fulfilled and the traces to $u$ exist.

\end{example}

\section*{Acknowledgements}

The authors are immensely grateful to the referee for his patience
and help. Without his comments and remarks the paper would hardly be
completed.

The research is done under the DAAD
Stability Pact for South Eastern Europe, Center of Excellence for Applications of Mathematics.

The work of J. Aleksi\' c is partially supported by Ministry
              of Education and Science, Republic of Serbia,
              project no. 174024.

              Darko Mitrovic is engaged as a part time researcher at
              the University of Bergen. The position is financed by
              the Research Council of Norway whose support we
              gratefully acknowledge.

\end{document}